\documentclass[12pt]{amsart}

\usepackage{amssymb}
\usepackage{amscd}
\usepackage{epsfig}

\usepackage[bookmarks=true]{hyperref}
\hypersetup{colorlinks,citecolor=blue,linkcolor=blue}

\mathchardef\mhyphen="2D

\newtheorem{theorem}[subsection]{Theorem}
\newtheorem{lemma}[subsection]{Lemma}
\newtheorem{prop}[subsection]{Proposition}

\newtheorem{main}{Theorem}

\theoremstyle{definition}
\newtheorem{example}[subsection]{Example}
\newtheorem{rem}[subsection]{Remark}
\newtheorem{question}[subsection]{Question}
\newtheorem*{question*}{Question}

\DeclareMathOperator{\SL}{SL}
\DeclareMathOperator{\GL}{GL}
\DeclareMathOperator{\PGL}{PGL}

\DeclareMathOperator{\SO}{SO}

\DeclareMathOperator{\PSL}{PSL}

\def\XX{\mathcal{X}}
\def\FF{\mathcal{F}}
\def\kk{\mathfrak{k}}
\def\Hy{\mathcal{H}}
\def\HH{{H}}

\def\G{\mathrm{G}}
\def\M{\mathrm{M}}

\def\cO{\mathcal{O}}
\def\cP{\mathcal{P}}
\def\disc{\mathrm{D}}
\def\reg{\mathrm{R}}
\def\dim{\mathrm{dim}}
\def\rank{\mathrm{rank}}
\def\e_HW{\epsilon_{HW}}

\def\Z{{\mathbb Z}}
\def\Q{{\mathbb Q}}
\def\R{{\mathbb R}}
\def\C{{\mathbb C}}

\def\F{{\mathbb F}}

\newcommand\area{\mathrm{area}}
\newcommand\vol{\mathrm{vol}}
\newcommand\len{\ell}
\newcommand\absys{\mathrm{absys}}
\newcommand\sys{\mathrm{sys}}

\newcommand\Rrk{{\mathbb{R}\mhyphen\mathrm{rank}}}

\def\An{\mathrm{A}}
\def\Bn{\mathrm{B}}
\def\Cn{\mathrm{C}}
\def\Dn{\mathrm{D}}
\def\En{\mathrm{E}}
\def\Fn{\mathrm{F_4}}
\def\Gn{\mathrm{G_2}}

\begin{document}
\title{Growth of $k$-dimensional systoles in congruence coverings}

% \date{\today}

\author{Mikhail Belolipetsky}
\thanks{Belolipetsky is partially supported by CNPq and FAPERJ research grants.}
\address{
IMPA\\
Estrada Dona Castorina, 110\\
22460-320 Rio de Janeiro, Brazil}
\email{mbel@impa.br}

\author{Shmuel Weinberger}
\thanks{Weinberger is partially supported by NSF grant DMS-2105451.}

\address{
%Department of Mathematics \\
University of Chicago \\
5734 University Ave. \\
Chicago, IL 60637 \\
USA}
\email{shmuel@math.uchicago.edu}

\begin{abstract}
We study growth of absolute and homological $k$-dimensional systoles of arithmetic $n$-manifolds along congruence coverings. Our main interest is in the growth of systoles of manifolds whose real rank $r \ge 2$. We observe, in particular, that in some cases for $k = r$ the growth function tends to oscillate between a power of a logarithm and a power function of the degree of the covering. This is a new phenomenon. We also prove the expected polylogarithmic and constant power bounds for small and large $k$, respectively.
\end{abstract}

\maketitle

\section{Introduction}\label{sec1}

Consider a semisimple Lie group $H$ without compact factors and corresponding symmetric space $\XX$. For example, we can take $H = \SL_2(\R)$, $\SL_3(\R)$, $\SL_2(\R)\times\SL_2(\R)$, etc. with the corresponding symmetric spaces --- the hyperbolic plane $\Hy^2$, the space of $3\times 3$ positive-definite matrices of determinant $1$, the product space $\Hy^2\times\Hy^2$. These first few examples will, in fact, already suffice for introducing the main questions and results of the paper. Let $M = \Gamma\backslash \XX$ be an arithmetic $\XX$-locally symmetric space and $\{M_i \to M\}_{i = 1,2,\ldots}$ a sequence of its regular \emph{congruence coverings} of degrees $d_i$. The construction of such sequences is well known, we will review it in Section~\ref{sec2.1}.
Denote by $\absys_k$ the \emph{absolute $k$-dimensional systole} of $M$, which is defined as the infimum of $k$-dimensional volumes of the subsets of $M$ which cannot be homotoped to $(k-1)$-dimensional subsets in $M$ (see~\cite{Gromov_Luminy92}). It is clear that $\absys_n(M_i) = \vol(M_i) = \vol(M) d_i$ and it is well known, although less obvious, that $\absys_1(M_i)$ has the same degree of growth as $c_1 \log(d_i)$ for a positive constant $c_1$ (see e.g. \cite[Proposition~16]{GuthLub15}). We are interested in understanding the growth rate of the intermediate systoles $\absys_k(M_i)$, $k = 2, \ldots, n-1$, in such sequences. The same problem can be considered for more general expanding families of non-positively curved manifolds or orbifolds, for example, for the sequences studied by Raimbault in \cite{Raimb13}. We expect that similar phenomena will occur in the general setting but postpone its investigation for the future.

The real rank of a symmetric space $\XX$ is equal to the maximal dimension of a connected, totally geodesic, flat submanifold in $\XX$. This also defines the $\Rrk(M)$ of the locally symmetric spaces covered by $\XX$ and the real rank of the associated Lie group $H$. The spaces and Lie groups with real rank greater than or equal to $2$ are called \emph{higher rank} and, although the statements of some results will cover the rank $1$ spaces as well, we will be mainly interested in the higher rank case. The question about the growth of $k$-dimensional systoles in congruence coverings was first raised by Gromov (see \cite[pp. 338--339]{Gromov_Luminy92}). He indicated, in particular, that one might expect to have the polylogarithmic (i.e. a polynomial in the logarithm) growth of $\absys_k(M_i)$ for $k\le \Rrk(M)$ and a constant power growth for bigger $k$. Our main result is as follows.

\begin{main}\label{thm A}
Let $M$ be a compact arithmetic locally symmetric space of real rank~$r$, $\{M_i \to M\}$ a sequence of congruence coverings of degrees $d_i$, and let $r_1 = r_1(M) \le r$ be a parameter of the root system associated to the algebraic group $\G$ defining $M$. Then $\absys_k(M_i)$ grows polylogarithmically with $d_i$ for $1\le k \le \max\{1,r_1\}$ and as a power function for $r < k \le n$.
\end{main}

The number $r_1$ is called the \emph{strongly orthogonal rank} of $M$. It is defined in Section~\ref{sec:sor} where we also study its basic properties. For example, the groups $\SL_{n+1}(\R)$ of type $\An_{n}$ and real rank $n$ 
have the strongly orthogonal rank equal to $[\frac12(n+1)]$, while the split orthogonal groups $\SO(n+1,n)$ of type $\Bn_n$ 
have the real rank and the strongly orthogonal rank both equal to $n$. A classical theorem of Borel \cite{Borel63} shows that these groups do contain cocompact arithmetic lattices and thus Theorem~\ref{thm A} applies to the corresponding locally symmetric spaces.

The precise statement of the theorem with additional information about the intermediate range and the implied constants will be given in Section~\ref{sec:absys}. 

Between the ranks $r_1(M)$ and $r(M)$, we have no reason to believe that generally polylogarithmic growth persists (although it does not persist beyond that).  Assuming some number theoretic conjectures about regulators of number fields, it should ``occasionally'' be that small.  We expect, again on the basis of the generic behavior of regulators that is explained in Section~\ref{rem:regulators}, that the constant power growth also occurs in this range.
However, because it is possible that the smallest cycles do not necessarily have an arithmetic source, we cannot be sure that this is true of the non-arithmetic cycles, as well. 
The existence of the intermediate range is a new phenomenon which does not feature in the conjectural picture suggested by Gromov in \cite{Gromov_Luminy92}. It can be observed only for symmetric spaces of certain types. For example, the spaces corresponding to higher rank simple Lie groups of type $\An$ always have this property while the spaces corresponding to simple split groups of type $\Bn$ do not have it.
One of our main goals in this paper is pointing out the provocative possibility of a variety of behaviors in the intermediate range.

Together with the absolute systoles $\absys_k(M)$ we can also consider \emph{homological systoles} of the congruence coverings. Recall that following Berger the \emph{$k$-dimensional systole} $\sys_k(M; A)$ is defined as the infimum of the $k$-dimensional volumes of the $k$-cycles in $M$ with coefficients in $A$ which are not homologous to zero in $M$. In recent years there has been a lot of interest in studying the rank of homology groups of arithmetic manifolds and its growth along the congruence coverings. We refer to \cite{Calegari15} for a related discussion and references. Our perspective here is different: we focus our attention on the metric properties of non-trivial cycles and extremal representatives. In Section~\ref{sec:homsys} we prove bounds for homological systoles of the congruence coverings with the coefficients in $\Q$, $\Z$, and $\Z_p$. 

In Section~\ref{sec:surfaces} we discuss in more detail $2$-dimensional systoles in higher rank manifolds and geometry of  representing surfaces. We observe that a recent work of Long and Reid \cite{LongReid19} implies that higher rank manifolds of large injectivity radius may contain $\pi_1$-injective surfaces of small genus. This provides
a striking contrast to the rank one case for which it was shown in \cite{Bel13} that if injectivity radius grows to infinity, then the genera of $\pi_1$-injective surfaces grow to infinity too. The induced metric on the small genus surfaces is (close to) a flat metric with singular points. We observe that the surfaces constructed by Long and Reid for $H = \SL_3(\R)$ fell into the intermediate range of Theorem~\ref{thm A} and discuss their possible connection to the absolute systoles of congruence coverings. Indeed, it is the existence of these strange surfaces that prevents us from proving a definite relation between the absolute systole in the intermediate range and regulators of number fields.

\subsection*{Notation}
We are using the topologists' notation $\Z_n = \Z/n\Z$, in particular, $\Z_p$ denotes the integers mod $p$, not the $p$-adic integers.

% The relation $f(x) \gtrsim g(x)$ for two positive functions $f(x)$ and $g(x)$ means that for any $\epsilon > 0$ there exists $x_0 = x_0(\epsilon)$ such that for all $x\ge x_0$ we have $f(x) \ge (1-\epsilon)g(x)$. We write $f(x) \sim g(x)$ if $f(x) \gtrsim g(x)$ and $g(x) \gtrsim f(x)$.

\section{Preliminaries}\label{sec2}

This paper requires a considerable amount of material from the theory of algebraic groups and related topics. We refer to Humphrey's book \cite{Humph75} as a basic reference. For the exposition more focused on arithmetic subgroups we refer to Witte Morris \cite{WitMor15}. An excellent introduction to geometry of locally symmetric spaces is provided by Eberlein's book \cite{Eber-book}.

\subsection{Arithmetic subgroups of semisimple Lie groups}\label{sec2.1} 
Let $H$ be a semisimple connected linear Lie group without compact factors. Assume that there exists a simply connected algebraic group $\G$ defined over a number field $K$ which admits an epimorphism $\phi:\G(K\otimes_\Q\R)^o \to H$ with a compact kernel. Let us fix a $K$-embedding $\G\hookrightarrow \GL_n$ and define the corresponding group of the integral points $\Gamma = \G(\cO)$ to be $\G(K)\cap \GL_n(\cO)$, where $\cO$ denotes the ring of integers of the field $K$. Then, by the Borel--Harish-Chandra theorem \cite{BorHC62}, $\phi(\G(\cO))$ is a finite covolume discrete subgroup of $H$. Such groups and all the subgroups of $H$ which are commensurable with them are called \emph{arithmetic lattices} (or arithmetic subgroups), and the field $K$ is called their field of definition. Choosing different admissible fields $K$ and groups $\G$ we can produce different commensurability classes of arithmetic subgroups of $H$. It is worth noting that for many groups $H$ this is the only known construction of lattices (i.e. discrete subgroups of finite covolume). In this paper we will always consider \emph{irreducible arithmetic lattices} (a lattice $\Gamma$ in $H$ is called \emph{irreducible} if $\Gamma N$ is dense in $H$ for every non-compact closed normal subgroup $N$ of $H$). By a fundamental theorem of Margulis all irreducible lattices in a higher rank Lie group are arithmetic (see \cite[Theorem~1$'$, p.~4]{Marg91}).  

Given an ideal $\cP$ of $\cO$, the \emph{principal congruence subgroup of level $\cP$} of an arithmetic group $\Gamma = \G(\cO)$ is defined to be $\Gamma(\cP) = \Gamma \cap \GL_n(\cO, \cP)$, where $\GL_n(\cO, \cP)$ denotes the subgroup of matrices in $\GL_n(\cO)$ which are congruent to the identity matrix modulo $\cP$. We define the principal congruence subgroups of $H$ to be the images of such principal congruence subgroups under $\phi$. An arithmetic subgroup which contains some principal congruence subgroup is called a \emph{congruence subgroup}. A celebrated conjecture of Serre asserts that all irreducible arithmetic subgroups of higher rank Lie groups are congruence.

To a lattice $\Gamma$ in $H$ we can associate a \emph{locally symmetric space} $\Gamma\backslash \XX$, where $\XX = H/K$ is the symmetric space of $H$. In this paper, we are interested in geometry of the locally symmetric spaces associated to the sequences of congruence subgroups. Many basic facts about arithmetic subgroups and their locally symmetric spaces can be found in \cite{WitMor15} and the references given there.

\subsection{A Riemannian metric on locally symmetric spaces} Let $H$ be a semisimple real Lie group as above and let $K$ denote its maximal compact subgroup.
Denote by $\mathfrak{h}$ and $\mathfrak{k}$ the Lie algebras of $H$ and $K$, respectively. Since $H$ is semisimple, the Killing form $B(x,y) = \mathrm{Tr}(\mathrm{Ad}(x)\mathrm{Ad}(y))$ is a nondegenerate symmetric bilinear form on $\mathfrak{h}$. Let $\mathfrak{p}$ denote the orthogonal complement of $\mathfrak{k}$ in $\mathfrak{h}$ with respect to the form $B(x,y)$, so that $\mathfrak{h} = \mathfrak{k} \oplus \mathfrak{p}$ is a Cartan decomposition of $\mathfrak{h}$. The restriction of the Killing form to $\mathfrak{p}$ induces a Riemannian metric on the symmetric space $\XX = H/K$ and the $\XX$-locally symmetric spaces. This metric has non-positive sectional curvature.
We refer to \cite{Helg01} for the details and more properties of this metric.

Given a locally symmetric space $M = \Gamma\backslash H/K$, the closed geodesics of $M$ correspond to semisimple elements in $\Gamma$. Following Prasad--Rapinchik, the length of the closed geodesic corresponding to a primitive semisimple element $\gamma\in \Gamma \subset \G(K)$ is given by the formula (\cite[Proposition~8.5]{PR09}):
\begin{equation}\label{length formula}
\len(\gamma) = \left( \sum_{\alpha\in\Phi(\G, \mathrm{T})} (\log |\alpha(\gamma)|)^2 \right)^{1/2} 
\end{equation}
where the summation is over all roots of $\G$ with respect to $\mathrm{T}$, a maximal $\R$-torus of $\G$ containing $\gamma$, and log denotes the natural logarithm. 

% where $V_K$ denotes the set of the archimedean valuations of the field $K$, $\mathrm{T}$ is a maximal $\R$-torus of $\G$ containing $\gamma$, and $\Phi(\G, \mathrm{T})$ denotes the absolute root system of $\G$ relative to $\mathrm{T}$ (see \cite[Section~8]{PR09}). % \textbf{Double-check the details about the length formula.}

\subsection{Systoles} \label{sec:prelim-systoles}
Let $M$ be a closed $n$-dimensional Riemannian manifold of non-positive sectional curvature. In this paper we will consider only locally symmetric manifolds $M$ endowed with the Riemannian metric described above but the following definitions are more general.

Following Gromov \cite{Gromov_Luminy92}, we define the \emph{absolute $k$-dimensional systole} $\absys_k(M)$ as the infimum of $k$-dimensional volumes of the subsets $C$ of $M$ which cannot be homotoped to $(k-1)$-dimensional subsets in $M$. Here a ``subset'' means a piecewise smooth subpolyhedron in $M$. % The cellular approximation theorem implies that this quantity is independent of a CW structure on $M$. By the Whitney Approximation Theorem we can approximate $C$ by a smooth immersion of a closed Riemannian manifold homotopic to $C$. By the Eells--Sampson Theorem \cite{EelSam64} this immersion can be homotoped to a harmonic map, the minimal volume is always attained on a harmonic immersion (see \cite[Section~2(D)]{EelSam64}), and therefore the systole is attained on a harmonic branched immersion.  

Notice that the absolute systole can be infinite. For instance, for the sphere $\mathcal{S}^n$ we have $\absys_k(\mathcal{S}^n) = \infty$ for all $0 < k < n$. Elementary homological algebra implies that for the locally symmetric spaces the situation is different:

\begin{prop}
Let $M$ be a $K(\Gamma, 1)$-space. Then $\absys_k(M) < \infty$ if and only if  $k\leq \mathrm{cd}(\Gamma)$, the cohomological dimension of $\Gamma$.
\end{prop}

\begin{proof}
If $\absys_k(M) = \infty$, then any rectifiable $k$-skeleton of $M$ would be homotopic to the $(k-1)$-skeleton $M^{k-1}$,
which would imply that for any coefficient system (i.e. $\Z\Gamma$-module) $B$ the injective restriction
map $\HH^k(M; B) \to \HH^k (M^k; B)$ factors through $\HH^k(M^{k-1}; B) = 0$. It shows that for all $B$, $\HH^k(M; B) =
0$, a characterization of $\mathrm{cd}(\Gamma)<k$ (see \cite[page 185]{Brown}).
\end{proof}

Unfortunately, nothing says that these classes are detected by any “conventional” coefficient
system, e.g. a finite dimensional local system – so there probably are examples where there are
no cycles in any conventional sense that detect $\absys_k(M)$.

For locally symmetric manifolds $M$ with a lower bound on injectivity radius one can easily get that $\absys_k(M) \leq c\,\vol(M)$ for all $k$. Moreover, in many cases it is possible to see that in covers (or even in general for manifolds of large volume) we have $\absys_k(M) = o\big(\vol(M)\big)$ --- for example it happens when there are nontrivial geodesic cycles to blame. We obtain the results of this kind in Theorem~\ref{thm3.1}. One might speculate that this is true quite generally, however, when there are no nontrivial geodesic cycles the situation can be markedly different (cf. Remark~\ref{rem4.6}). 

We now proceed with the definition of homological systoles. 
The \emph{$k$-dimensional systole} $\sys_k(M; A)$ is the infimum of the $k$-dimensional volumes of the $k$-cycles in $M$ with coefficients in $A$ which are not homologous to zero in $M$. This classical definition goes back to Berger. For convenience, we say that $\sys_k(M; A) = \infty$ if $\HH_k(M;A)$ is trivial.

It follows immediately from the definitions that for every $k$ we have:
\begin{equation}\label{eq:absys-sys}
 \absys_k(M) \leq \sys_k(M; A).
\end{equation}

 We refer to Gromov's lectures \cite{Gromov_Luminy92} for a comprehensive introduction to systoles and intersystolic inequalities.
 
\subsection{The real rank and flat subspaces} \label{sec:prelim-regulators}
In this paper we are mainly interested in the higher rank symmetric spaces. The distinctive feature of these spaces is the presence of totally geodesic \emph{flats}, i.e. totally geodesic subspaces isometric to a flat Euclidean space. The maximal dimension of the flats is equal to the real rank of $\XX$. These flats give rise to the flat tori in our manifolds $M_i$ whose volumes (with respect to the induced metric) are related to the arithmetic invariants of $M_i$. We now briefly review this relation.  

Let $F$ be a maximal compact flat of dimension $r$ in an arithmetic manifold $M = \Gamma\backslash \XX$. Let us note for the future reference that $r$ is equal to the real rank of the Lie group $H$, which is defined as the dimension of a maximal $\R$-split torus in $H$. The description of flat subspaces of $M$ was given by Mostow if $M$ is compact and by Prasad--Raghunathan in general \cite{Most73, PrRag72}. They are obtained as
$$ F = Z\backslash \FF,$$
where $\FF$ is a maximal flat in $\XX$ and $Z = Z_\Gamma(\gamma)$ is the centralizer of an $\R$-regular (and $\R$-hyper-regular if $\Gamma$ is non-uniform) element $\gamma \in \Gamma$. We recall that an element $\gamma$ is called \emph{$\R$-regular} if the number of eigenvalues, counted with multiplicity, of modulus $1$ of $\mathrm{Ad}(\gamma)$ is minimum possible.
The $\R$-regular elements are those for which the centralizer in $H$ of the polar part has the smallest possible dimension (see~\cite[Section~2]{Most73}).

If $\Gamma$ is a congruence subgroup of a sufficiently large level, we can assume that $\Gamma \subset \GL_s(\cO)$ and hence an $\R$-regular element $\gamma$ is represented by a regular matrix $A$ with the entries in $\cO$. It is well known that in this case the centralizer $Z_{\mathrm{M}_s(K)}(\gamma)$ can be identified with the ring of polynomials in $A$ with coefficients in $K$, and hence with a subring of the field extension $L = K[X]/(P(X))$, where $P(X)$ is the minimal polynomial of $A$ (it is irreducible because $A$ is a regular matrix). Under this isomorphism the elements from $Z_\Gamma(\gamma)$ correspond to the relative units in $L$ (but not all the relative units give rise to the elements in $\Gamma$). 

The infinite part of the group of units of the number field $L$ is generated by independent fundamental units $u_1, \ldots, u_r$, and to each of them we can assign a vector 
$$(N_1\log|\sigma_1(u_i)|, \ldots,  N_d\log|\sigma_d(u_i)|),$$ 
where $\sigma_1$, \dots, $\sigma_d$ denote the archimedean places of $L$ and $N_i = 1$ or $2$ depending on $\sigma_i$ being real or complex. The absolute value of the determinant of an $r\times r$ submatrix of the matrix composed by these vectors does not depend on the choice of the submatrix and is called by the \emph{regulator} of the field $L$. Up to a scaling factor, it is equal to the volume of the $r$-dimensional parallelepiped spanned by the vectors. We refer to Lang's algebraic number theory \cite[Chanter~V]{Lang86} for the classical theory of units and regulators. A similar construction extends to the units in the fields extensions and relative regulators introduced by Berg{\'e} and Martinet in \cite{BM87}. In our case this translates to the relation between the relative regulator and the volume of the associated torus. It can be formally verified by repeating the argument in the proof of \cite[Proposition~8.5(ii)]{PR09} applied to the generators of $Z_\Gamma(\gamma)$, which can be diagonalized simultaneously because they commute. We summarize this discussion in the following statement: 

\begin{prop}
The volume of an $r$-dimensional flat torus $F$ in a real rank $r$ arithmetic manifold $M = \Gamma\backslash \XX$ which is associated to an $\R$-regular (and $\R$-hyper-regular if $\Gamma$ is non-uniform) element $\gamma \in \Gamma$ satisfies the inequality
\begin{equation}\label{eq:vol_F}
	\vol_r(F) \ge c\cdot \reg_{L/K},
\end{equation}
where the fields $K$ and $L$ are defined as above, the relative regulator satisfies $\reg_{L/K} \geq \reg_L/\reg_K$ by its definition in \cite{BM87}, and $c = c(\XX)$ is a positive constant which depends on the normalization of the metric on $\XX$. 
 
\end{prop} 

In some cases it is also possible to bound $\vol_r(F)$ by the regulator from above. For example, this can be done if $\Gamma = \G(\cO_K)$. However, when we descend along a sequence of subgroups $\{\Gamma_i\}$ certain regular elements $\gamma_i \in \Gamma_i$ will give the same field extensions $L/K$ with the same regulator. The volumes of the corresponding flat subspaces $F_i$ in $\XX/\Gamma_i$ are proportional to $\reg_{L/K}\cdot [Z_\Gamma(\gamma_i):Z_{\Gamma_i}(\gamma_i)]$ and grow with the index. 
Frequently these volumes will be greater than the systole and hence we will not need to consider them, but there are also some cases in which they may dominate the systole growth (cf. the second upper bound inequality in Theorem~\ref{thm3.1}).

\medskip

We conclude with a basic illustrative example.

\begin{example}
Let $H = \SL_3(\R)$ and let $\gamma$ be given by the matrix
$$A = 
\begin{pmatrix}
	0 &  0 & 1 \\
	1 &  0 & 4 \\
	0 &  1 & 3 \\
\end{pmatrix}.$$
Its characteristic polynomial $P(X) = -X^3 + 3X^2 + 4X +1$ and the eigenvalues are approximately $-0.69202$, $-0.35690$, and $4.0489$. It follows that $\gamma$ is $\R$-regular, moreover, it is $\R$-hyper-regular (in the case of $\SL_n(\R)$ the hyper-regularity condition from \cite{PrRag72} means that there are no relations $\lambda_1^{m_1} \cdots \lambda_n^{m_n} = 1$ with positive integer exponents $m_i$ among the eigenvalues $\lambda_i$ of the matrix except when $m_1 = \ldots = m_n$). Notice that if $\lambda$ is a root of $P(X)$, then it is a unit in $\Z[\lambda]$, and $\lambda+1$ is a root of $P(X-1) = -X^3 + 6X^2 - 5X + 1$ and hence is also a unit. It is easy to check that they are independent fundamental units of $L = \Q[X]/(P(X))$. 
The associated $\log$-vectors are approximately given by 
\begin{align*}%\label{log-vecs}
(\log(0.69202), \log(0.35690), \log(4.0489)) &= (-0.36814, -1.0303, 1.39845), \text{ and }\\
(\log(0.30798), \log(0.64310), \log(5.0489)) &= (-1.17772, -0.441455, 1.61917).
\end{align*}
Hence the regulator $\reg_L$ given by the absolute value of the determinant of a $2\times 2$ submtrix of the matrix formed by these vectors is approximately equal to $1.05088$. The fact that $\reg_L \neq 0$ confirms that the units $\lambda$ and $\lambda+1$ are multiplicatively independent.

Consider the flat subspace $F$ of $M = \SO(3)\backslash \SL_3(\R)/\SL_3(\Z)$ given by the centralizer $Z(\gamma)$ in $\SL_3(\Z)$. We have $Z(\gamma) = \langle A, A+\mathrm{Id}\rangle$, as the matrices $A$ and $A+\mathrm{Id}$ are associated to the fundamental units $\lambda$ and $\lambda + 1$ in the field $L$. We are interested in the volume of the subspace $F$ in $M$. Following the previous discussion, it is given by the area of the parallelogram spanned by the two $\log$-vectors in $\R^3$, which is equal to $\sqrt{3}\cdot\reg_L$.   
\end{example}

\section{Strongly orthogonal rank}\label{sec:sor}

We proceed with defining the parameter $r_1 = r_1(M)$, which we call the \emph{strongly orthogonal rank} of $M$. Assume $M = \Gamma\backslash\XX$ and the arithmetic group $\Gamma$ is associated to an algebraic $K$-group $\G$. Let $\Psi$ represent a set of homomorphisms $\rho_\alpha : \mathrm{A}_\alpha \to \G$ defined over $K$ such that for each $\alpha$ we have:
\begin{itemize}
 \item[-] the group $\mathrm{A}_\alpha$ is an almost simple algebraic $K$-group of type $\An_1$;
 \item[-] for every archimedean place $v: K\to \C$ for which $\G(K_v)$ is not compact the group $\mathrm{A}_\alpha(K_v)$ is also not compact;
 \item[-] the kernel of $\rho_\alpha$ is contained in the center of $\mathrm{A}_\alpha(K)$; and
 \item[-] the images of different $\rho_\alpha$ mutually commute and intersect trivially in $\G(K)$.
\end{itemize}
The strongly orthogonal rank $r_1(M)$ is the maximal cardinality of such a set $\Psi$.

\medskip

Following Tits \cite{Tits66}, algebraic groups of type $\An_1$ defined over a number field $K$ fell into three types associated to quadratic forms, Hermitian forms, and quaternion algebras, respectively. Moreover, in this case all the three types can be obtained as groups of units of quaternion algebras over $K$ (see \cite[Chapters~2 and 7]{MaclReid03} for a detailed description). Over the reals the group $\mathrm{A}_\alpha(K\otimes_\Q\R)$ is isogenous to $\SL_2(\C)^a\times\SL_2(\R)^b\times \SO(3)^c$ with $2a + b + c = [K:\Q]$. 

\medskip

Our notion of strong orthogonality is a generalization of strongly orthogonal root systems. Recall that a subset $\Delta$ of positive roots in an irreducible finite root system $\Phi$ is called a \emph{strongly orthogonal subset} if $\alpha \pm \beta \not\in \Phi\cup\{0\}$, for any two roots $\alpha, \beta \in \Delta$, and a subset $\Delta$ is called a maximal strongly orthogonal if it is strongly orthogonal and is not properly contained in any other strongly orthogonal subset. In particular, the condition $\alpha \pm \beta \not\in \Phi\cup\{0\}$ implies that any two roots in $\Delta$ are orthogonal to each other with respect to the inner product defined by the Killing form. If the group $\G$ associated to $M$ is a split $K$-group, then the sets $\Psi$ defined above are given by the strongly orthogonal subsets of the $K$-root system of $\G$ (see more details in the proof of Proposition~\ref{prop:sor} below). 

The strongly orthogonal root systems previously appeared in various sources. They were first introduced by Harish-Chandra in his work on representations of semisimple Lie groups \cite[Part~II, Section~6]{HC56}.  The maximal orders $N(\Phi)$ of the strongly orthogonal root systems were computed by Hee Oh in \cite{Oh98}, they are given in Table~\ref{table:sor}. Notice that for the types  $\Bn_n$, $\Cn_n$, $\Fn$, $\Gn$, $\En_7$, and $\En_8$ we have $N(\Phi) = \rank(\Phi)$.  

\medskip
In general setting, the homomorphisms $\rho_\alpha$, $\alpha \in \Psi$, define a family of Lie subalgebras $\mathfrak{r}_\alpha$ of the Lie algebra $\mathfrak{g}$ of the algebraic group $\G$ (see \cite[Chaper~III]{Humph75}). Since the images of different $\rho_\alpha$ mutually commute, the commutators of the elements from different $\mathfrak{r}_\alpha$ are zero. Hence different $\mathfrak{r}_\alpha$ are orthogonal $K$-subspaces of $\mathfrak{g}$  with respect to the inner product defined by the Killing form. Therefore, considering $\mathfrak{g}$ as a vector space over $K$, we can choose a basis of  $\mathfrak{g}$ such that the subspaces $\mathfrak{r}_\alpha$ are linearly independent. We now choose an embedding of $\G$ into $\GL_n$ compatible with this $K$-structure. This defines the group of integral points of $\G$ up to commensurability. 
% In the following sections, when considering a rational embedding of a $K$-group $\G$ into $\GL_n$, we can always assume that it is chosen in this way. In particular, this choice of embedding is relevant to the proof of Lemma~\ref{lem41} in Section~\ref{sec:homsys}.       

\begin{table}[ht]
\begin{center}
\def\arraystretch{1.2}
  \begin{tabular}{l|l|l}
   \hline
    $N(\Phi)$ & Type of $\Phi$ & Examples of groups\\ \hline
    $[\frac12(n+1)]$ & $\An_n$ & $\SL_{n+1}(\R)$\\
    $n$ & $\Bn_n$ & $\SO(n+1,n)$ \\
    $n$ & $\Cn_n$ & $\mathrm{Sp}(2n,\R)$ \\
    $2[\frac12 n]$ & $\Dn_n$ & $\SO(n,n)$ \\
    $4$ & $\En_6$ & \\
    $\mathrm{rank}(\Phi)$ & $\Fn$, $\Gn$, $\En_7$, $\En_8$ & \\
   \hline
  \end{tabular}
\vskip2em
    \caption{Maximal orders $N(\Phi)$ of the strongly orthogonal root systems (here $[x]$ denotes the largest integer less than or equal to $x$).}
    \label{table:sor}
\end{center}
\end{table}

\begin{example} Let $f = x_1^2 + x_2^2 - \sqrt{2}x_3^2 + x_4^2 - \sqrt{2}x_5^2$ be a quadratic form in $\R^5$ of signature $(3,2)$. It is defined over the field $K = \Q(\sqrt{2})$ and the Galois conjugate form $f^\sigma$ for the non-identity embedding $\sigma: K\to\R$  is positive definite. Therefore, the group $\G = \SO(f;K)$ is an anisotropic $K$-group and there are no strongly orthogonal $K$-roots. On the other hand, the strongly orthogonal rank $r_1(\G) = 2$ is equal to the real rank of the symmetric space of $\G(\R)$. The associated $K$-homomorphisms $\rho_1, \rho_2: \G_1 = \SO(g;K) \to \G$, with $g = y_1^2 + y_2^2 - \sqrt{2}y_3^2$, are given by the maps of $y_1, y_2, y_3$ to $x_1, x_2, x_3$ and $x_1, x_4, x_5$, respectively.
\end{example}

We proceed with some basic properties of the strongly orthogonal rank.

\begin{prop}\label{prop:sor}
Given an algebraic $K$-group $\G$ with the real rank $r$ and the strongly orthogonal rank $r_1$, we have:
\begin{enumerate}
 \item The value $r_1$ is bounded below by the maximal cardinality of a strongly orthogonal subset in the $K$-root system of $\G$;
 \item It is bounded above by the real rank, i.e. $r_1 \le r$;
 \item If $K\mhyphen\mathrm{rank}(\G) = r$, then $r_1$ is equal to the maximal order of the strongly orthogonal root system of $\G$;
 \item There are examples of groups with $r_1 = 0$ and groups with $r_1 = r$.
\end{enumerate}
\end{prop}
\begin{proof}
(1) Recall that two positive roots $\alpha$ and $\beta$ in a nonmultipliable $K$-root system $\Sigma$ of $\G$ are called strongly orthogonal if neither of $\alpha\pm\beta$ is a root. Following \cite[3.1.(13)]{Tits64}, to each $\alpha \in \Sigma$ we can associate a $K$-homomorphism $\rho_\alpha: \SL_2 \to \G$ whose kernel is contained in the center of $\SL_2$ and which has a prescribed action on the generators of $\SL_2(K)$. The Chevalley commutator relations imply that if the roots $\alpha$ and $\beta$ are strongly orthogonal, then the images of $\SL_2(K)$ under the corresponding homomorphisms commute. Hence a maximal set of pairwise strongly orthogonal roots provides an admissible set $\Psi$ of homomorphisms that define the strongly orthogonal rank of $\G$.

\medskip

(2) Given a $K$-homomorphism $\rho_\alpha : \mathrm{A}_\alpha \to \G$ with $\alpha \in \Psi$, by taking the tensor product with $\R$ over $\Q$ we obtain a homomorphism
$$\tilde\rho_\alpha: \mathrm{A}_\alpha(K\otimes_\Q \R) \to G = \G(K\otimes_\Q \R).$$
The image of $\tilde\rho_\alpha$ is noncompact hence it contains a $1$-dimensional $\R$-split torus $T_\alpha$ of $G$. By the definition of the strongly orthogonal rank, for different $\alpha, \beta \in \Psi$ the subgroups $T_\alpha$ and $T_\beta$ commute in $G$. It follows that the product $\prod_{\alpha\in\Psi}T_\alpha$ is an $\R$-split torus of $G$ of dimension $r_1$ and thus $r_1 \le r$, the maximal dimension of such a torus.

\medskip

(3) If $\Rrk(\G) = K\mhyphen\mathrm{rank}(\G)$, then $\G(K\otimes_\Q \R)$ has no non-trivial compact factors, and hence for any $\alpha \in \Psi$ the group $\mathrm{A}_\alpha(K\otimes_\Q \R)$ has no compact factors either. It follows that $\mathrm{A}_\alpha$ is isogenous to $\SL_2$ over $K$ and we can define an induced $K$-homomorphism $\rho'_\alpha: \SL_2 \to \G$. This allows us to associate to any $\alpha\in \Psi$ a positive root in the $K$-root system of $\G$, and then to associate to $\Psi$ a strongly orthogonal system of roots proving that the maximal cardinality of such a system is $\ge r_1$. The opposite inequality was already shown in (1).

\medskip

(4) Let $D$ be a division algebra of prime degree over $K$ and let $\G = \SL_1(D)$. The group $\G$ has no $K$-defined subgroups except for the $K$-tori and hence $r_1(\G) = 0$. (This example was suggested to us by Andrei Rapinchuk.) For the case $r_1 = r$ we can consider $K$-split groups of types $\Bn$ or $\Cn$ and apply a result of Hee Oh about the cardinality of the strongly orthogonal root systems of split groups \cite{Oh98} (see also Remark~\ref{rem_OrtRoots} in the next section).
\end{proof}

\section{Growth of absolute systoles}\label{sec:absys}

In this section we prove Theorem~\ref{thm3.1} which implies Theorem~\ref{thm A} from the introduction.

\begin{theorem}\label{thm3.1}
Let $M$ be a compact arithmetic $\XX$-locally symmetric space defined over a number field $\kk$ and let $\{M_i \to M\}$ be a sequence of its regular congruence coverings of degrees $d_i$. Assume that the real rank of $\XX$ is $r$ and let $r_1' = \max\{1, r_1(M)\}$, where the strongly orthogonal rank $r_1(M)$ is defined as above.  
\begin{itemize}
\item[(i)] There exist positive constants $a_k$, $b_k$, $\beta_k$, $\gamma_k$ depending on $M$ such that for sufficiently large $d_i$ we have
\begin{align*}
a_k \log^k(d_i) &\le  \absys_k(M_i) \le b_k \log^k(d_i), \textrm{ for }1\le k \le r_1';\\
a_k \log^k(d_i) &\le  \absys_k(M_i) < b_k d_i^{\gamma_k}, \textrm{ for }r_1' < k \le r;\\
a_k d_i^{\beta_k} &\le  \absys_k(M_i) \le  b_k d_i^{\gamma_k}, \textrm{ for }r < k \le n.
\end{align*}
\item[(ii)]
For $r_1' < k \le r$, the constant $\gamma_k < 1$ and for the principal congruence coverings we can take $\gamma_k = \frac13$ if $\XX$ is not of the type $\An_1$ and $\gamma_k = \frac13 + \epsilon$, $\epsilon > 0$ for the type $\An_1$. If for $r < k < n$ some $M_{i_0}$ contains a $k$-dimensional closed totally geodesic submanifold, then $\gamma_k < 1$ in this dimension while otherwise the corresponding $\gamma_k = 1$.
\end{itemize}
\end{theorem}

\begin{proof} The argument naturally splits into three parts corresponding to the different ranges of $k$.
	
{\bf 1.} Let $1 \leq k \leq r_1'$.
The lower bound for $k = 1$ follows from the well known fact that the $1$-dimensional systole of $M_i$ is bounded below by $c\log(d_i)$ (see \cite[Proposition~16]{GuthLub15}). Together with Gromov's systolic inequality for essential polyhedra from  \cite[Appendix~2, Theorem~B$'_1$]{Gromov83} this implies the lower bound for the other $k$ (see also \cite[Section~3.C.9]{Gromov_Luminy92}). We remark that in the terminology that we use in this paper, Gromov's ``essential'' means \emph{absolutely essential} as these polyhedra cannot be homotoped to smaller dimensional subsets.  

For the upper bound let us first consider the case when $r_1'$ is equal to the cardinality of a strongly orthogonal subset in the $\kk$-root system $\Sigma$ of the group $\G$ associated with $M$. Let $\Sigma^+$ denote the subset of positive roots. By \cite{Tits64}, for each root $\alpha\in \Sigma$ which is not twice another $\kk$-root there is a homomorphism $\rho_\alpha : \SL_2 \to \G$ defined over $\kk$ whose kernel is contained in the center of $\SL_2(\kk)$. Denote by $G_\alpha$ the image of $\SL_2(\kk)$ under $\rho_\alpha$, it is a subgroup of $\G(\kk)$ isomorphic to $\SL_2(\kk)$ or $\PSL_2(\kk)$.

It follows from the Chevalley commutator relations that the strong orthogonality of two roots $\alpha$ and $\beta \in \Sigma^+$ is equivalent to the condition that $x_\alpha x_\beta = x_\beta x_\alpha$ for any $x_\alpha \in G_\alpha$, $x_\beta \in G_\beta$. Moreover, as $\alpha$ and $\beta$ are different roots, the non-trivial elements $x_\alpha$ and $x_\beta$ are not proportional. It follows that whenever $x_\alpha$ and $x_\beta$ have infinite orders they generate a free abelian subgroup of $\G(\kk)$ of rank $2$.

Now when we descend along the sequence of the normal congruence subgroups of $\G(\kk)$, they intersect the subgroups $G_\alpha$ in their normal congruence subgroups, which implies that $G_\alpha \cap \Gamma(\cP_i)$ will contain hyperbolic elements whose displacements are proportional to $\log(d_i)$. Such elements taken from the subgroups that correspond to the different strongly orthogonal roots would commute, hence they generate a free abelian subgroup of full rank of the fundamental group of $M_i$. By a theorem of Gromoll and Wolf \cite{GrWolf71}, this implies that $M_i$ contains a flat torus whose dimension $k$ is equal to the number of generators and whose $k$-volume is proportional to $\log^k(d_i)$. This gives the upper bound for the absolute $k$-systole for $k \le r_1'$. 

In order to remove the assumption on $r_1'$, consider the homomorphisms $\rho_\alpha : \mathrm{A}_\alpha \to \G$ from the definition of the strongly orthogonal rank. The commutation and independence of their images is part of the assumption and the rest of the argument applies directly. 

\medskip

{\bf 2.} Consider the intermediate range $r_1' < k \le r$. As in part 1, the lower bound follows from the logarithmic bound on $\absys_1(M_i)$ and Gromov's systolic inequality \cite[Theorem~0.1.A]{Gromov83}.

For the upper bound we recall the well known results of Gromoll--Wolf \cite{GrWolf71} and Prasad--Raghunathan \cite{PrRag72}, which imply that for $k\le r$ the manifolds $M_i$ contain totally geodesic, immersed, $k$-dimensional flat tori. Consider one of these tori in some $M_i$. By \cite[Theorem~2.13 and discussion in the next paragraph]{PrRag72}, there exists an algebraic $\kk$-torus $\mathrm{H}$ in $\G$ associated with it. Let $\mathcal{F}$ be the flat subspace of $\XX$ corresponding to $\mathrm{H}$, $Z = \Gamma\cap\mathrm{H}(\kk)$, and $\{ Z(\cP_i) \}$ are the congruence subgroups of $Z$ which all act discretely and cocompactly on $\mathcal{F}$. By using the strong approximation property we can compare the growth rates of the covering degree $d_i$ and the degree of the covers $Z\backslash\FF \to Z(\cP_i)\backslash\FF$. We give more details about this argument below. It follows that the volumes of the tori $Z(\cP_i)\backslash\FF$ in $M_i$ grow as $b_k d_i^{\gamma_k}$ with $\gamma_k < 1$. This gives the upper bound for the growth of the absolute $k$-systoles. 

For the explicit upper bound we can use a result of Vdovin \cite{Vdovin99}: %and/or \cite{Vdovin01}.
If $G = \G(\F_q)$ is a finite simple group not of the type $\An_1$, the order of any abelian subgroup $A < G$ is less than $|G|^{1/3}$ (while for the type $\An_1$ this would be an asymptotic upper bound when $q \to \infty$). Given our sequence of principal congruence subgroups $\Gamma(\cP_i)$ of $\Gamma$, by the strong approximation theorem (recall that the group $\G$ in the definition of arithmetic subgroups is simply connected so the strong approximation property holds, see~\cite[Section~7.4]{PlaRap94}), for sufficiently large $i$ we have:
$$
\begin{CD}
1 @>>> \Gamma(\cP_i) @>>> \Gamma @>>> \G(\F_{q_i}) @>>> 1 \\
@. @AAA @AAA @AAA @.\\
1 @>>> Z(\cP_i) @>>> Z @>>> A  \\
\end{CD}
$$
(Here $\{\cP_i\}$ is a sequence of ideals of the ring of integers $\cO$ of $\kk$ and $\F_{q_i} = \cO/\cP_i$.)

This implies (for $\G$ a simple group not of type $\An_1$):
\begin{align*}
[Z:Z(\cP_i)] & \le |A| < |\G(\FF_{q_i})|^{1/3}; \\
\vol(Z(\cP_i)\backslash\FF) &< \vol(Z\backslash\FF) \cdot |\G(\F_{q_i})|^{1/3} \leq \vol(Z\backslash\FF) \cdot d_i^{1/3}.
\end{align*}
Notice that the strong approximation property providing surjectivity of the map $\Gamma \to \G(\F_{q_i})$ is used in the last inequality. It relates the order of $\G(\F_{q_i})$ to the degree $d_i$ of the covering and is crucial for our argument.

It is easy to extend this result to a more general case of semisimple Lie groups and deduce that
$\vol(Z(\cP_i)\backslash\FF) < Cd_i^{1/3}$ for $\G$ not of the type $\An_1$, while for the type $\An_1$, given any $\epsilon > 0$ and $d_i$ large enough, $\vol(Z(\cP_i)\backslash\FF) < Cd_i^{1/3+\epsilon}$.

% Remark: can try to use non-principal congruence subgroups here to get a faster systole growth.

\medskip

{\bf 3.} Now let $r < k \le n$. Following \cite[Section~2]{Gromov83}, we recall the notion of the filling volume of a singular cycle in a complete metric space $\mathcal{M}$. First define the volume of a singular simplex $\sigma: \Delta^k \to \mathcal{M}$ as the lower bound of the total volumes of those Riemannian metrics on $\Delta^k$ for which the map $\sigma$ is distance decreasing. Then we also have the notion of the volume of a singular chain $C = \sum_i r_i\sigma_i$, namely, $\vol_k(C) = \sum_i |r_i|\vol_k(\sigma_i)$, where the coefficients $r_i$ may be real numbers, integers or residues mod $2$. Now for a $(k-1)$-dimensional singular cycle $Q$ in $\mathcal{M}$ we define the \emph{filling volume} of $Q$ as the lower bound of the volumes of those $k$-dimensional chains $C$ in $\mathcal{M}$ for which $\partial C = Q$.

By \cite[Theorem~1(ii)]{Leu14} (see also \cite[Section~$5.D(5)(b')$]{Gromov93}), for $k > r$ the symmetric space $\XX$ satisfies the linear isoperimetric inequality:
\begin{equation}\label{isoperim_ineq}
\mathrm{FillVol}_k(Q) \le c\,\vol_{k-1}(Q),
\end{equation}
where $c>0$ is a constant, $Q$ is any $(k-1)$-dimensional cycle, and $\mathrm{FillVol}$ denotes the filling volume of $Q$. Although the theorem is stated for cycles, the proof applies more generally to bound volume of a $k$-dimensional cone over a $(k-1)$-dimensional chain in $\XX$. 

Let $N$ be a Lipschitz $k$-chain nearly representing $\absys_k(M_i)$, i.e. we have $\vol(N) \leq \absys_k(M_i) + \varepsilon$ and $N$ can not be homotoped to a $(k-1)$-dimensional subpolyhedron in $M_i$. We consider $N$ with the induced metric.  
Let $\mathrm{B}_{p,N}(R) \subset N$ be a metric ball of radius $R$ with center $p \in N$ and denote $v_p(R) = \vol_k(\mathrm{B}_{p,N}(R))$. 

For any $R < R_i/2$, where $R_i$ is the injectivity radius of $M_i$, the boundary $S = \partial\,\mathrm{B}_{p,N}(R)$ is a $(k-1)$-chain in $M_i$ to which the isoperimetric inequality \eqref{isoperim_ineq} applies and the filling volume of this chain is $\geq v_p(R)-\varepsilon$ because $N$ is an almost systole. 
In more detail, consider the inclusion of  the chain $N$ into $M_i$.  Let $Y = (N \smallsetminus \mathrm{B}_{p,N}(R)) \cup Q$ (where $Q$ is the cone of $S$, as above).  The ball $\mathrm{B}_{p,N}(R)$ in $N$ has a map into $Q$ extending the identity map on $S$ (as $Q$ is contractible).  In the universal cover $\XX$, the lifts of the inclusion map and the map into $Q$ are homotopic rel boundary (again by contractibility, now of $\XX$).  As a result, the inclusion map of $N$  in $M$ is homotopic to a map lying in $Y$, so the inclusion of $Y$ in $M$ is not null-homotopic.
Combining these observations with the coarea formula (see \cite[Theorem~13.4.2]{BurZalg}), we have:
\begin{align*}
v_p(R) & = \int_0^R \vol_{k-1}\left(\partial\,\mathrm{B}_{p,N}(t)\right) dt \\
       & = v_p(R_0) +  \int_{R_0}^R \vol_{k-1}\left(\partial\,\mathrm{B}_{p,N}(t)\right) dt \\    
      & \geq v_p(R_0) + \int_{R_0}^R \frac1c (v_p(t) - \varepsilon)dt\ \text{ (by \eqref{isoperim_ineq});}\\
% v'(R) & \geq \frac1c v(R);\\
v_p(R) - v_p(R_0) & \geq   C_1(e^{R/c} - e^{R_0/c}),\ C_1 > 0,  
\end{align*}
for $R > R_0$ which we think of as a fraction of the injectivity radius of the base manifold $M$. 
For this argument to apply we have to ensure that $v_p(R_0) > \varepsilon$. 

We now show that $C_p = (v_p(R_0) - \varepsilon)$ is strictly positive at some point $p$. To this end we use a powerful deformation theorem of Federer and Fleming which allows us to deform general chains to the simplicial ones in a controlled way \cite{FedFl60}. Federer and Fleming’s theorem works for normal currents but as it is explained in \cite[Section~3]{CDMW18} it also holds for the Lipschitz chains. Our application, indeed, is similar to the ones considered in \cite{CDMW18}. 

First we triangulate the manifold $M$ at the scale $R_0 =  \mathrm{RadInj}(M)/10$ and lift the triangulation to the covering $M_i$. Using the Federer--Fleming argument we can deform the chain $N$ via random projection to the $k$-skeleton.  If the projection does not fill some cell, deform that image into the $(k-1)$-skeleton. Thus the definition of the absolute systole shows that some cell must be filled. The deformation theorem implies that a point in the preimage of that cell has ball of radius $R_0$ with a bounded below volume. Therefore, for sufficiently small $\varepsilon = \varepsilon(M) > 0$, we have $C_p > 0$.  

Since $2R_i = \absys_1(M_i) \ge c_2\log(d_i)$, we obtain that for $k > r$, 
$$\absys_k(M_i) \ge a_k d_i^{\beta_k},$$
with the positive constants $a_k$, $\beta_k$ depending on $M$.

\medskip

For the general upper bound in this range we recall that if $M$ is an aspherical manifold, the $k$-skeleton is essential in $M$ and the linear upper bound follows (cf. \cite[Section~3.C]{Gromov_Luminy92}). 

Now assume that $M$ (or some cover $M_{i_0}$) contains a $k$-dimensional finite volume totally geodesic arithmetic submanifold $N$. By a result of Bergeron--Clozel the manifold $N$ is arithmetic \cite[Proposition~15.2.2]{BC05} (see also \cite[Theorem~1.7]{BBKS}). Note that the field of definition of $N$ can be an extension of $\kk$, however, there is a $\kk$-group $\mathrm{H}$ associated to $N$ by Bergeron and Clozel which is a proper algebraic $\kk$-subgroup of $\G$ such that the group $\Lambda$ defining $N$  is contained in its $\cO$-points. We denote by $\mathcal{U}$ the totally geodesic symmetric subspace of $\XX$ corresponding to $N$.

Applying the strong approximation property similar to the way it was done in part~2 of the proof we have:
$$
\begin{CD}
1 @>>> \Gamma(\cP_i) @>>> \Gamma @>>> \G(\F_{q_i}) @>>> 1 \\
@. @AAA @AAA @AAA @.\\
1 @>>> \Lambda(\cP_i) @>>> \Lambda @>>> \mathrm{H}(\F_{q_i})  \\
\end{CD}
$$
(as before, $\{\cP_i\}$ is a sequence of ideals of the ring of integers $\cO$ of $\kk$ and $\F_{q_i} = \cO/\cP_i$). 

The dimension of the group $\mathrm{H}$ is strictly smaller than $\dim(\G)$, therefore, we obtain 
$$\vol(\Lambda(\cP_i)\backslash\mathcal{U}) \leq \vol(\Lambda\backslash\mathcal{U}) |\mathrm{H}(\F_{q_i})| \leq 
\vol(\Lambda\backslash\mathcal{U}) |\G(\F_{q_i})|^{\gamma_k} \leq c\,d_i^{\gamma_k},\textrm{ with } \gamma_k < 1.$$
This finishes the proof.
\end{proof}

We proceed with some remarks about the theorem.

\subsection{}\label{rem_OrtRoots}
We see that for certain simple Lie groups we may have $r_1 = r$. For example, this occurs for the split groups of the types that are listed in the last row of Table~\ref{table:sor}. Thus, in these cases there is no intermediate range $r_1 < k \le r$. From the other hand, for the irreducible lattices in semisimple product Lie groups $r_1$ is fixed while $r$ grows with the number of factors. The difference between $r_1$ and $r$ can also attain arbitrarily large values for the simple groups of type $\An$.

\subsection{}\label{rem:regulators}
The systole behavior in the intermediate range $r_1' < k \le r$ is related to the flat subspaces whose volumes depend on the regulators of certain extensions of the field of definition. For $k = r$ this is explained in Section~\ref{sec:prelim-regulators}, and a similar description extends to the other $k$.

From number theory we know the following properties of the regulator $\reg_K$:

\begin{itemize}
 \item[(a)] E. Landau \cite{Land18}: $\reg_K \le c_1(n) \sqrt{\disc_K} \log^{n-1}(\disc_K)$;
 \item[(b)] J. Silverman \cite{Silv84}: $\reg_K \ge c_2(n) \log^{r(K)-\rho(K)}(\gamma(n)\disc_K)$,
\end{itemize}
where $D_K$ denotes the absolute value of the discriminant of the field $K$, $n$ is the degree of $K$, $r(K)$ is the rank of the unit group of $K$, $\rho(K)$ is the maximum of $r(L)$ over the proper subfields $L \subset K$, and $c_1(n)$, $c_2(n)$ and $\gamma(n)$ are positive constants which depend only on the degree of the field. 

It is expected that as a function of the discriminant the regulator oscillates wildly between these bounds while staying closer to the upper bound for most of $\disc_K$. This conjecture is completely open even for the sequences of quadratic number fields. We refer to the beautiful Lenstra's article \cite{Lenstra08} for an introduction and more details about this topic. For us the number theoretic conjecture indicates that it may be possible to find congruence towers with a constant power systole growth in the intermediate range.

Geometrically regulator is the volume of a higher dimensional parallelepiped span\-ned by the vectors associated to the independent fundamental units of a field. It appears that it is often impossible to choose these units to be simultaneously small which forces the regulator to be bigger than the lower bound. Similarly, volumes of the flat tori of sufficiently large dimension in a locally symmetric space are hard to get small because when we choose some meridians short, the others are forced to be long. This observation can be made more precise by exploiting the connection with number theory. For $r$-dimensional tori this is done in Section~\ref{sec:prelim-regulators}. We do not pursue the details regarding the  intermediate dimensions $k$ but we notice that when $k \leq r_1'$ the situation changes and we can always choose small generators for the flat $k$-cycles. This is shown in part~1 of the proof of Theorem~\ref{thm3.1}.  

\subsection{}\label{rem:regulators-example} 
We can make an experiment with $\Gamma = \SL_3(\Z)$, representing perhaps the simplest non-trivial case. The quotients here will be non-compact finite volume manifolds. Although in this paper we mostly consider compact arithmetic manifolds, we expect that analogous results should hold for the non-compact finite volume spaces. The assumption that the base field $K = \Q$ simplifies the arithmetic of the computations.

Consider the matrices of the form
$$
\begin{pmatrix}
    1+p^2 &  p & p^2 \\
      p   &  1 &  p \\
      0   &  p & 1+p^2 \\
\end{pmatrix},
$$      
where $p$ is a prime number.
They correspond to hyperbolic elements $\gamma_p \in \Gamma(p)$. By formula \eqref{length formula}, the length of the  geodesic in $M_p = \Gamma(p)\backslash \XX$ corresponding to $\gamma_p$ is proportional to $\log(p)$. We now consider the area of the associated tori.

The characteristic polynomial of $\gamma_p$ is $1 - (3 + 5p^2 + 4p^4)x + (3 + 5p^2)x^2 - x^3$. It has discriminant $D_p = 144 p^{12} + 520 p^{10} + 793 p^8 + 500 p^6$ and 3 different real roots (all not equal to $\pm 1$). It follows that the elements $\gamma_p$ are $\R$-regular, moreover, they are $\R$-hyper-regular for all primes $p$ (recall that in the case of $\SL_n(\R)$  hyper-regularity means that there are no relations $\lambda_1^{m_1} \cdots \lambda_n^{m_n} = 1$ with positive integer $m_i$ among the eigenvalues $\lambda_i$ of the matrix except when $m_1 = \ldots = m_n$).

We used GP Pari to compute the regulators of the associated fields $L$ for the first 1000 primes and then produced a plot with Mathematica. The resulting graph is shown in Figure~\ref{fig1}. As expected,  the regulator oscillates between a logarithmic lower bound and a constant power upper bound.  

\begin{figure}[!ht]
\includegraphics[width=15cm]{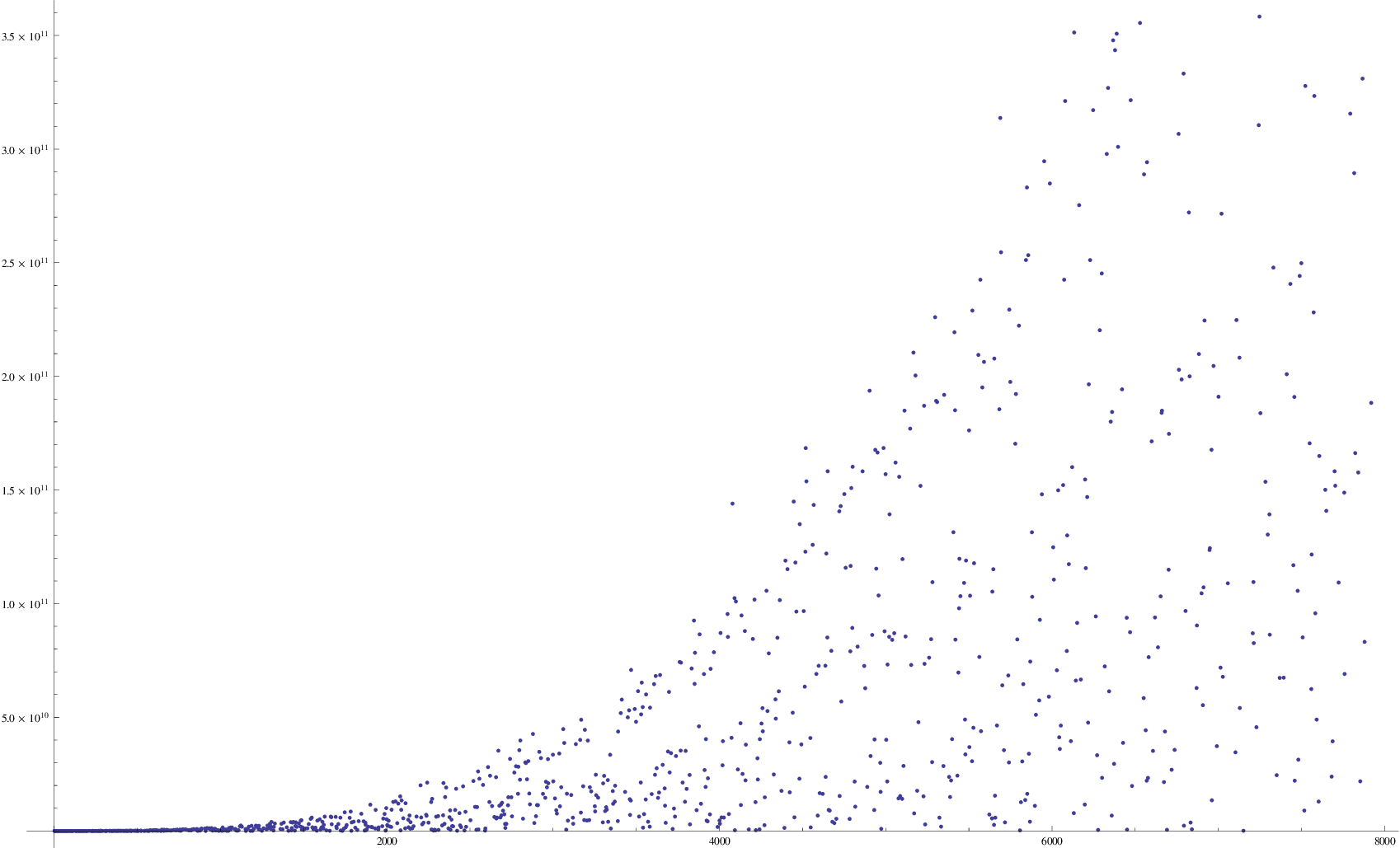}
\caption{}
\label{fig1}
\end{figure}

It is likely that both logarithmic lower bound and power function upper bound are attained for appropriate subsequences of primes, but based on what we know from number theory any result of this kind would be very hard to prove. Next, in order to obtain the volumes of the corresponding tori in $M_p$ we need to multiply the regulator by the index $[Z_{\Gamma}(\gamma_p) : Z_{\Gamma(p)}(\gamma_p)]$. It is not clear whether after this correction the (conjectural) logarithmic lower bound will persist, if not, we will have the volumes of the tori oscillating between two power functions.

In the above experiment, we considered a particular sequence of geodesics of length $c \log(p)$ in the congruence coverings $M_p \to M$ and observed that associated regulators tend to be large and oscillate. A natural question arising here is what can we say about the volumes of the tori associated with other choices of the closed geodesics in $M_p$. It is a priori possible that for some other sequence of $\gamma_p$ we would get much smaller volumes of the tori, pushing the $2$-systole bound of $M_p$ down to a logarithmic function. This behavior seems unlikely but more research is required in order to be able to rule it out. We are planning to set up a thorough computational experiment addressing this problem in a forthcoming project.

It is worth mentioning that there has been a large body of recent work about distribution of periodic torus orbits on the higher rank locally symmetric spaces. We can refer, in particular, to \cite{ELMV09} and the subsequent papers by those authors. Einsiedler, Lindenstrauss, Michel, and Venkatesh use a different notion of the discriminant of a periodic orbit which they introduce in \cite{ELMV09}, however, it can be shown that their notion is related to the discriminant of the field $L$ defined in Section~\ref{sec:prelim-regulators}. Among the other results, it was proved in \cite{ELMV09} that periodic torus orbits in $\PGL_3(\Z)\backslash\PGL_3(\R)$ have the following properties:
\begin{enumerate}
\item The volume of a periodic orbit of discriminant $D$ is bounded, up to constants, between $\log(D)$ and $D^c$ for $c > 0$; % (see Proposition 2.8).
\item The number of periodic torus orbits of discriminant at most $D$ is bounded, up to constants, between $D^A$ and $D^B$ for $A,B > 0$; and % (see Proposition 2.9).
\item The orbits of small volume close to $\log(D)$ come in packets with large multiplicity of order about $D^{1/2+o(1)}$, where all orbits in a packet share the same discriminant, volume and even shape.
\end{enumerate}
The latter property is particularly interesting from our viewpoint as it indicates the existence of small tori with large multiplicity in the congruence coverings.

\subsection{}\label{rem4.5} 
Let us remark that the quantitative values for the exponent $\gamma_k$ in part~(ii) of the theorem require the assumption that the congruence subgroups are \emph{principal}. In particular, it is possible to construct sequences of non-principal congruence coverings for which the corresponding bound for the constant $\gamma_k$ will be bigger and the growth of $\absys_k(M_i)$ can be faster than for the principal congruence subgroups. This is interesting because of the connection between the growth of systoles and quantum error-correcting codes, with the faster growth rate corresponding to better codes. We refer to \cite{GuthLub15} for more details. In particular, the upper bound of order $n^{0.3}$ from \cite[Remark~20]{GuthLub15} changes to $n^{0.5}$ if instead of the congruence kernels $\Gamma_N$ considered there we take the preimages of the unipotent radicals of the Borel subgroups. It is not clear how close this bound is to the actual asymptotic growth of systoles but notice that we still get the corresponding exponent $\gamma_k < 1$. 

\subsection{}\label{rem4.6} 
It is known that the manifolds $M_i$ do not necessarily contain $k$-dimensional finite volume totally geodesic submanifolds. For example, arithmetic hyperbolic manifolds which are not of the simplest time do not have codimension $1$ totally geodesic subspaces (see e.g. \cite{BBKS}). The notion of \emph{fc-subspaces} introduced in \cite{BBKS} can be applied to effectively rule out totally geodesic subspaces of certain sufficiently large dimensions in other examples of arithmetic manifolds. Can we expect a much faster growth of $\absys_k$ in these cases? This leads to the following open problem:
\begin{question*} 
Does there exist a sequence of congruence coverings of an arithmetic $n$-dimensional manifold whose $\absys_k$ grows linearly with the degree for some $k < n$?
\end{question*}

\subsection{}\label{rem4.7} 
It would be interesting to know more about the constants $a_k$, $b_k$ and their dependence on the symmetric space $\XX$ and the locally symmetric space $M$. In particular, we can ask about the best \emph{asymptotic values} of the constants. Besides of the trivial case $\absys_n(M_i) = \vol_n(M) d_i$, these values are known for $\absys_1$ of  arithmetic hyperbolic $2$-manifolds~\cite{BS94}, arithmetic hyperbolic $3$-manifolds~\cite{KSV07}, Hilbert modular varieties \cite{Mur17}, arithmetic hyperbolic $n$-manifolds of the simplest type~\cite{Mur19}, and quaternionic hyperbolic manifolds \cite{EKM22}.

\section{Homological systoles}\label{sec:homsys}

Let $A$ denote the ring $\Q$, $\Z$ or $\Z_q$, where $q$ a prime number. (The results of this section also apply to the case $A = \F_q$, a finite field with $q$ elements where $q$ is a prime power.) The \emph{$k$-dimensional systole} $\sys_k(M; A)$ is defined as the infimum of the $k$-dimensional volumes of the non-trivial cycles in $\HH_k(M; A)$. We would like to investigate the behavior of the homological systole $\sys_k(M_i; A)$ in the sequences of congruence coverings $\{M_i \to M\}$ of an arithmetic manifold $M$.

The main principle here is the following: Whenever we have a non-trivial cycle in $\HH_k(M; A)$, the corresponding lower bound from Theorem~\ref{thm3.1} applies for $\sys_k(M; A)$. This follows immediately from inequality \eqref{eq:absys-sys} in Section~\ref{sec:prelim-systoles}. Moreover, when we know the origin of some non-trivial cycles in $M$ we can often prove that the corresponding upper bound applies as well. This, however, depends on a concrete situation. In particular, by Borel's stability theorem \cite{Borel74} the groups $\HH_k(M, \Q)$ are trivial for certain $k$ and hence we cannot expect the two-sided inequalities for the homological systole to hold in general.

Over the last few years several new interesting results on non-trivial cycles in homology of arithmetic locally symmetric manifolds were published. In particular, we call attention to the papers \cite{Avr15, Bena21, Zschumme21} whose constructions we will use in this section. For the basic facts in algebraic topology that are frequently used in this section we refer to Hatcher's book~\cite{Hatcher-AlgTop}.

\medskip

We now prove two lemmas that capture the bottom and the top of the systolic spectrum.

\begin{lemma}\label{lem41}
Let $M_i$ be a cover of a compact higher rank irreducible arithmetic manifold $M$ which corresponds to a congruence subgroup $\Gamma(\cP_i)$ of  $\pi_1(M) = \Gamma < \G(\kk)$ of sufficiently large level $\cP_i$ with $\cO/\cP_i \cong \F_{q_i}$. Define $r_1' = r_1'(M)$ as in Theorem~\ref{thm3.1}. Then for $1 \le k \le r_1'$ the group $\HH_k(M_i; \Z)$ contains a non-trivial element 
(detected by mod $q_i$ cohomology) and we have
$$ a_k \log^k(q_i) \le  \sys_k(M_i; A) \le b_k \log^k(q_i),$$
where $a_k$, $b_k$ are positive constants depending on $M$ and $A = \Z$ or $\Z_{q_i}$.
\end{lemma}

\begin{proof}
The basic case is $k = 1$. Let $\cP = \cP_i$ be the level of the congruence subgroup $\Gamma(\cP_i)$, $\cO/\cP \cong \F_{q}$, and assume that $\Gamma(\cP)\subset \GL_m(\cO)$. In order to produce a non-trivial cycle in $\HH_1(M_i; \Z_q)$ consider the map
\begin{align*}
\phi: \Gamma(\cP)   &\to \M_m(\F_{q}),\\
 \phi(I + \cP C) &= [C]_\cP,
\end{align*}
where $[C]_\cP$ denotes the projection of the matrix $C$ to $\M_m(\F_{q})$. This map is a homomorphism from $\Gamma(\cP)$ to the additive group of $m\times m$ matrices over $\F_{q}$. Indeed, given $B_1 = I + \cP C_1$ and $B_2 = I + \cP C_2$, we have:
$$\phi(B_1B_2)  = \phi(I + \cP(C_1 + C_2) + \cP^2C_1C_2) = [C_1 + C_2]_\cP = \phi(B_1) + \phi(B_2).$$
By the strong approximation property \cite[Section~7.4]{PR09}, for almost all $\cP$ the image of the map $\phi$ is a non-trivial subspace of $\Z_q^{m^2} \cong \M_m(\F_{q})$ (indeed the image is the Lie algebra of $\G$ in the Lie algebra of $\SL_n$). Hence the induced map $\phi_*: \HH_1(M_i; \Z_q) \to \HH_1(\Z_q^{m^2}; \Z_q)$ has a non-zero image, which implies that $\HH_1(M_i; \Z_q)$ is non-trivial. This argument is similar to the one used for $\Gamma = \SL_m(\Z)$ in \cite[Section~2]{ChWe14}.

By Margulis's normal subgroup theorem the group $\HH_1(M_i; \Z)$ is finite (see \cite[Theorem~4, p.~3]{Marg91}), hence the universal coefficient theorem implies that $\HH_1(M_i; \Z)$ has an element of order divisible by $q$. 

The lengths of non-trivial cycles in $\HH_1(M_i; \Z)$ are bounded from below by the absolute $1$-systole. For the upper bound observe that every basis of $\HH_1(M_i; \Z)$ contains a cycle whose length is smaller than the diameter of $M_i$, and the latter is bounded above by $c \log\big(\vol(M_i)\big)$ by a theorem of Brooks \cite{Br92}.

Now let $1 < k \le r_1'$. We have a set of $\kk$-homomorphisms $\rho_1$,\ldots $\rho_k$, whose image generates a $\kk$-subgroup of $\G$ which we denote by $\G_1$. We restrict attention to this group and its congruence subgroups $\Gamma_1(\cP) =  \Gamma(\cP)\cap\G_1(\kk)$, $\cP = \cP_i$ all considered as subgroups of $\GL_m(\cO)$. Applying the map $\phi$ defined above to the group $\Gamma_1(\cP)$ we find $k$ non-trivial independent cycles $c_1,\ldots,c_k\in\HH_1(M_i; \Z)$ of orders divisible by $q$ that have mutually commuting preimages in $\Gamma(\cP)$ (cf. Part~1 of the proof of Theorem~\ref{thm3.1} and note that the Lie algebras of the images of the maps $\rho_j$ are orthogonal subalgebras of the Lie algebra of $\G$).
Applying the theorem of Brooks~\cite{Br92} to the quotient spaces associated to the homomorphisms $\rho_j$ we can choose the cycles $c_j$, $j = 1,\ldots, k$, so that their lengths are bounded above by $c\,\log\big(\vol(M_i)\big)$. % and identify them with the independent cycles $c_1,\ldots,c_k$. 
Let $\alpha_1,\ldots,\alpha_k$ be the dual cocycles in $\HH^1(M_i; \Z_q) = \mathrm{Hom}(\HH_1(M_i; \Z);\Z_q)$.

By a theorem of Gromoll and Wolf \cite{GrWolf71}, the independent commuting elements $c_1,\ldots,c_k \in \Gamma(\cP)$ span a $k$-dimensional flat torus $T$ in $M_i$. A basic calculation shows that the cup product $\alpha_1\cup\ldots\cup\alpha_k$ evaluated on the torus $T$ is non-trivial (see \cite[Example~3.11, p.~210]{Hatcher-AlgTop}). Therefore, the product is non-trivial in cohomology $\HH^k(M_i; \Z_q)$ and the torus is non-trivial in homology $\HH_k(M_i; \Z)$ (or $\HH_k(M_i; \Z_q)$, as is shown by the same argument). 

By the construction, the volume of $T$ is bounded above by $b_k\log^k(q_i)$. The lower bound for the volume of $T$ as well as for any other non-trivial $k$-cycle follows from inequality~\eqref{eq:absys-sys} and Theorem~\ref{thm3.1}.
\end{proof}

\begin{rem}
There is a substantial research on stability of homology of arithmetic groups with coefficients in a ring (see e.g. \cite{Calegari15} and the references therein). As a sample result we can mention Calegari's theorem that under certain conditions evaluates the groups $\HH_k(\Gamma_N(\mathfrak{p}^m); \Z_p)$, where $\Gamma_N(\mathfrak{p}^m)$ is a congruence subgroup of $\SL_N(\cO)$ and $\mathfrak{p}$ is a prime ideal in $\cO$ above a rational prime $p$ (see \cite[Theorem~5.7]{Calegari15} for a precise statement). Here $N$ has to be taken sufficiently large and one can make this condition explicit by using the previous work of Maazen \cite{Maazen79} and van der Kallen \cite{Kallen80}. Thus, by Theorem~4.11 from \cite{Kallen80}, when $\cO$ is a principal ideal domain we have  $N \geq 2k + 1$, so when $N$ is odd this degree $k$ agrees with the strongly orthogonal rank while when $N$ is even the degree is $r_1 - 1$ (cf. Table~\ref{table:sor}). The method of Lemma~\ref{lem41} does not allow us to recover the whole homology groups, but instead we can capture the special cycles about which we have more geometric information. In this way we are explicitly seeing a small part of Calegari's picture.
\end{rem}

\begin{lemma}\label{lem42}
Let $M$ be a closed orientable manifold and assume that $\HH_k(M)$ \linebreak
$\big(=\HH_k(M;\Z)\big)$ 
contains a non-trivial $q$-torsion. Then the groups $\HH_{n-k-1}(M; \Z_q)$ and $\HH_{n-k}(M; \Z_q)$ are non-trivial. 
%Moreover, if $\HH_1(M)$ contains a non-trivial $q$-torsion, $\HH_{n-1}(M; \Z_q)$ is also non-zero. 
The same results apply for non-orientable closed manifolds if $q = 2$. 
\end{lemma}

\begin{proof}
By the universal coefficient theorem, the map
$$\mathrm{Ext}(\HH_k(M), \Z_q) \to \HH^{k+1}(M; \Z_q)$$
is injective, and since $\HH_k(M)$ has a non-trivial $q$-torsion, we have $\mathrm{Ext}(\HH_k(M), \Z_q) \neq 0$ (indeed, $\mathrm{Ext}(\Z_m, \Z_n) \cong \Z_{(m,n)}$ for any $m,n \ge 2$). This shows that $\HH^{k+1}(M; \Z_q)$ is non-zero. 

By the Poincar\'e duality,
$$\HH_{n-k-1}(M; \Z_q) \cong \HH^{k+1}(M; \Z_q),$$
which proves the first assertion of the lemma.

For proving the non-triviality of $\HH_{n-k}(M; \Z_q)$ notice that $\mathrm{Tor}(\HH_k(M), \Z_q) \neq 0$ implies that $\HH_k(M; \Z_q) \neq 0$, and hence $\HH_{n-k}(M; \Z_q) \neq 0$ by the  Poincar\'e duality and the universal coefficient theorem (since $\Z_q$ is a field).   

Every manifold is $\Z_2$-orientable hence in this case we do not need the orientability assumption.
\end{proof}

We now consider in detail three examples of real rank $2$ spaces. 

\begin{example}\label{ex41}
Let $\XX = \Hy^2\times\Hy^2$ the symmetric space of $H = \SL_2(\R) \times \SL_2(\R)$ and let $M$ be a compact orientable irreducible arithmetic $\XX$-manifold defined over a field~$\kk$. We have
$$\dim(M) = 4\text{, the real rank }r(M) = 2\text{, and } r_1'(M) = r_1(M) = 1.$$
Consider a sequence of principal congruence coverings $\{M_i \to M\}$ of degrees $d_i$. The real rank of $H$ is $>1$, hence by the Margulis normal subgroup theorem \cite[Theorem~4, p.~3]{Marg91} the group $\HH_1(M_i; \Z)$ is finite.
% it has the Margulis superrigidity property, which implies that $\HH_1(M_i; \Z)$ is finite (see \cite[Chapter~16]{WitMor15}). 
We can apply Lemma~\ref{lem41} to confirm that for sufficiently large level $\HH_1(M_i; \Z)$ is non-trivial and then apply Lemma~\ref{lem42} to deduce non-triviality of $\HH_3(M_i; \Z_{q_i})$ and $\HH_2(M_i; \Z_{q_i})$, with $q_i$ defined by the level of $M_i$. Therefore, we have
\begin{align*}
a_1 \log(d_i) &\le  \sys_1(M_i; \Z_{q_i}) \le b_1 \log(d_i);\\
a_2 \log(d_i)^2 &\le  \sys_2(M_i; \Z_{q_i}) < b_2 d_i^{\gamma_2};\\
a_3 d_i^{\beta_3} &\le  \sys_3(M_i; \Z_{q_i}) \le  b_3 d_i.
\end{align*}

In dimension $2$ we can say more. If we consider a flat torus $T \subset M$, its universal covering $\widetilde{T}$ is a $2$-dimensional rational flat subspace of $\XX$ for which we can find a complementary $2$-dimensional rational flat intersecting $\widetilde{T}$ in a single point (see \cite{Zschumme21}). This implies that for sufficiently large $i$ the manifolds $M_i$ contain a couple of totally geodesic tori which intersect transversally in finitely many points with all the intersections having the same sign \cite[Proposition~7.15]{Zschumme21}.
It follows that the classes of these tori are non-trivial in $\HH_2(M_i; \Q)$ and  $\HH_2(M_i; \Z)$ (see~\cite{Schw_survey}). This observation implies
$$a_2 \log(d_i)^2 \le  \sys_2(M_i; A) < b_2 d_i^{\frac13+\epsilon}, \textrm{ for }A = \Q \textrm{ or } \Z.$$

A similar argument applies to $k$-systoles of the locally symmetric spaces associated to $\XX = (\Hy^2)^k$ for $k \ge 2$. The details of the proof of non-triviality of the flat cycles are given in \cite{Zschumme21}.
\end{example}

\begin{example}\label{ex42}
Let $H = \SL_3(\R)$, $\XX$ its symmetric space, and $M$, $M_i \to M$ are associated compact arithmetic manifolds defined as in the previous example. We have
$$\dim(M) = 5,\ r(M) = 2\text{, and } r_1'(M) = r_1(M) = 1.$$
As before, using Lemmas~\ref{lem41} and \ref{lem42} we can show non-triviality of $\HH_k(M_i; \Z_{q_i})$ for $k = 1$, $3$, $4$ and sufficiently large $i$ and write the corresponding bounds for the homological systole. 

In dimension $k = 2$, the locally symmetric spaces $M_i$ have totally geodesic flat tori, however, the previous argument does not apply here and it is not clear how to check if they give rise to any non-trivial cycles. The case of non-compact $\XX$-manifolds was studied by Avramidi and Nguyen--Phan in \cite{Avr15}. They showed non-triviality of the flat cycles in $\HH_2(M_i; \Q)$ and raised the question about what happens in the compact case (see \cite[Section~9]{Avr15}).

Davis and Weinberger showed that no rational homology sphere of dimension $1$ mod $4$ can have a free $\Z_2\times\Z_2$ action \cite{Davis83, Weinb86}. Consequently, when taking a cover $M_i \to M$ corresponding to a group that contains $\Z_2\times\Z_2$, which holds for almost all principal congruence coverings, the manifold $M_i$ is not a rational homology sphere. Hence the Kazhdan Property~(T) (which implies that $H_1(M_i;\Q) = 0$) and the Poincar\'e duality show that there is non-trivial rational homology in dimensions $2$ and $3$ for this manifold. (Turning this argument into a quantitative one is a fascinating question.) 

%An implicit argument of Davis and Weinberger shows that a $5$-manifold $M$ with a free $\Z_2\times\Z_2$-action has even Euler semicharacteristic (see \cite[Section~8]{Davis83} and \cite[Section~4]{Weinb86}). The condition of the free action can be easily verified for the congruence coverings. Hence we have the \emph{Euler semicharacteristic}
%$$\chi_{\frac12}(M_i; \Q) := \rank(H_0(M_i;\Q)) - \rank(H_1(M_i;\Q)) + \rank(H_2(M_i;\Q))  \equiv 0\;(\mathrm{mod\ 2}).$$
%Considering that $\rank(H_0(M_i;\Q)) = 1$ and $\rank(H_1(M_i;\Q)) = 0$ (provided by the Kazhdan Property~T), it gives 
%$$\rank(H_2(M_i;\Q)) > 0.$$
%Assuming orientability of $M_i$, it also implies non-triviality of $\HH_3(M; \Q)$. Unfortunately, nothing is known about the nature of these non-trivial cycles but we do have lower bound for their volumes coming from the bounds for the absolute systoles.
\end{example}

\begin{example}\label{ex43}
Let $H = \SO(2,3)$ and $\XX$, $M$, $M_i \to M$ are the associated spaces defined as in the previous examples. We have
$$\dim(M) = 6,\ r(M) = 2\text{, and } r_1'(M) = 2.$$
Now Lemmas~\ref{lem41} and \ref{lem42} cover the dimensions $k = 1$, $2$, $4$, and $5$. Moreover, in dimension $3$ we have totally geodesic hyperbolic subspaces associated to $\kk$-rational embeddings $\SO(1,3) \hookrightarrow \SO(2,3)$. Similar to the discussion in Example~\ref{ex41} we can consider two complementary $3$-dimensional totally geodesic $\kk$-subspaces of $\XX$ which intersect in a single point, and then show that for sufficiently large $i$ the manifolds $M_i$ contain a couple of totally geodesic $3$-subspaces intersecting transversally in finitely many points with all the intersections having the same sign. This implies non-triviality of $\HH_3(M_i; \Q)$ for sufficiently large $i$. This construction fits into a general perspective developed by Millson and Raghunathan in \cite{MilRag81}
(see also \cite[Part~III]{Schw_survey} for an exposition). As before, non-vanishing of homology allows us to produce bounds for homological systoles. 
\end{example}

\section{Surfaces in higher rank manifolds}\label{sec:surfaces}

So far, while dealing with the upper bounds for systoles below the rank of the space we were mainly considering tori and corresponding cycles in homology. A natural question is whether there exist more complicated non-trivial topological subspaces that may have smaller induced volume. The basic case is about the $2$-dimensional systole in congruence coverings of higher rank arithmetic manifolds. We now look at it more carefully.

Let $M$ be a compact arithmetic locally symmetric manifold of real rank $r \ge 2$ and let $\{M_i \to M\}$ be a sequence of its regular congruence coverings of degrees $d_i$. We know that $\absys_1(M_i)$ grows to infinity. The following concrete question arises naturally:

\begin{question} \label{qq}
Does there exist a sequence of $\pi_1$-injective surfaces $S_i \hookrightarrow 
M_i$ such that all $S_i$ have bounded genus which is greater than 1?
\end{question}

Notice that the answer to this question is negative in the real rank one case where by \cite[Theorem~5.1]{Bel13} the genus is bounded below by an exponential function of the systole. Moreover, if such a sequence of surfaces exists, the metric on the minimal area representative of the homotopy class of $S_i$ will have to be (almost) flat at almost all points. The latter property can be checked by an argument similar to \cite[Proof of Theorem~C2]{Rez93}:

Let $\Sigma$ be a minimal area surface homotopic to $S_i$. It is a compact surface of genus $g$ with systole $s = \absys_1(\Sigma)$. Since $S_i$ is $\pi_1$-injective in $M_i$, we have $s \ge \absys_1(M_i)$ which grows to infinity with $i$.
By the Eells--Sampson theorem we have that $\Sigma$ is the image of a harmonic immersion of a Riemannian surface into $M_i$. 
Assume that the curvature $K(\Sigma') \le -\epsilon < 0$ at an open subset $\Sigma' \subset \Sigma$. By the Gauss--Bonnet theorem 
$$4\pi(g-1) \ge - \int_{\Sigma'} K dA \ge \epsilon\,\area(\Sigma').$$
Hence $\area(\Sigma') \le 4\pi(g-1)/\epsilon$ is  bounded while $\area(\Sigma) \to \infty$ with $i$. In fact, Gromov's systolic inequality for surfaces of genus $g > 1$ from \cite[Section~6.4]{Gromov83} implies that
$$\area(\Sigma) \ge C\frac{g}{\log^2(g)} \absys^2_1(M_i),$$
where $C > 0$ is an absolute constant.  

In a recent paper \cite{LongReid19}, Long and Reid were able to prove the first positive results towards Question~\ref{qq}. They showed, in particular, that for a lattice $\Lambda < \SL(3,\R)$ from a certain infinite set of pairwise non-commensurable uniform lattices, $\Lambda$ is commensurable with an infinite sequence of torsion-free congruence subgroups $\Gamma_j$ such that each $\Gamma_j$ contains a thin surface subgroup of a fixed genus. Here \emph{thin} means a Zariski dense subgroup of infinite index; this class of groups was introduced by Sarnak and subsequently has attracted a lot of interest (see~\cite{Sarnak14notes}).
The proof of Long and Reid is a nice geometric application of higher Teichm\"uller theory. It is worth mentioning that the key technical step of the argument in \cite{LongReid19} boils down to the solutions of Pell's equation.
Of course, this calls to mind Lenstra's article \cite{Lenstra08}, but we know of no relation between the areas of the Long--Reid surfaces and the regulators. 
In fact, we do not know any number theoretical invariant connected to these areas. One can speculate about a possible relation with higher regulators, but no results of this kind are available at this time. 

A particular class of surfaces that can be good candidates for geometric representatives of the Long--Reid construction are the genus $g > 1$ piecewise flat surfaces with a finite number of singular points. They will clearly satisfy the curvature conditions given above and can be realized by piecewise harmonic immersions. 

Notice that dimension of the Long--Reid surfaces matches precisely the conditions on the intermediate range for $\SL_3(\R)$: $r_1' < 2 \leq r$. It may occur that the $2$-dimensional absolute systole of the manifolds $M_j$ associated to $\Gamma_j$ is realized by such surfaces $S_j$, in which case we would like to have a control of their area. At this point, the available geometrical tools applied in the proof of Theorem~\ref{thm3.1} and Gromov's systolic inequality allow us to show that   
$$ a\,\frac{g_j}{\log^2(g_j)} \log\big(\vol(M_j)\big)^2 \le  \area(S_j) \le b\,\vol(M_j)^{\gamma},$$
where $g_j > 1$ is the genus of $S_j$ and $a$, $b$, $\gamma$ are positive constants with $\gamma < 1$.
The lack of a link to number theory prevents us from proving more. In this construction we would like to be able to compare the areas of the Long--Reid surfaces and the areas of the flat tori in $M_j$. The latter are controlled by the regulators and we have a conjectural picture about their oscillation along the congruence sequence. 

For the other groups Question~\ref{qq} remains open. In particular, it would be interesting to answer this question for the Hilbert modular varieties and for the compact irreducible manifolds covered by a product of hyperbolic planes. The method of \cite{LongReid19} falls short from covering these cases. 

\subsection*{Acknowledgements}
We thank Ted Chinburg, Manfred Einsiedler, Alex Eskin, Altan Erdnigor, Mikolaj Fraczyk, and Robert Young for helpful discussions related to this work.  This paper obviously owes a great deal to the ideas and questions raised by Gromov in \cite{Gromov_Luminy92}. We would like to thank the anonymous referees for their stimulating comments on the previous versions of this paper that surely have improved it in a number of ways. 

\bibliographystyle{amsplain}
\bibliography{mbel}

\end{document}